\documentclass[reqno]{amsart}

\usepackage[english]{babel}
\usepackage{amsthm,amsmath,amssymb}
\usepackage{mathtools}
\usepackage[
  pdfborder = {0 0 0},
  colorlinks = true,
  citecolor = black,
  linkcolor = black]{hyperref}
\usepackage{float}

\theoremstyle{plain}
\newtheorem{theorem}{Theorem}
\newtheorem{lemma}[theorem]{Lemma}
\newtheorem{corollary}[theorem]{Corollary}

\theoremstyle{definition}
\newtheorem{definition}[theorem]{Definition}

\newtheorem{remark}[theorem]{Remark}

\DeclareMathOperator{\Res}{Res}
\newcommand{\CC}{\widehat{\mathbb{C}}}

\title[Roots of the independence polynomial]{On the location of roots of the independence polynomial of bounded degree graphs}
\date{\today}
\author{Pjotr Buys}
\thanks{Funded by the Netherlands Organisation of Scientific Research (NWO): 613.001.851}
\email{pjotr.buys@gmail.com}

\begin{document}

\begin{abstract}
	In \cite{PetersRegts2017} Peters and Regts confirmed a conjecture by Sokal \cite{Sokal2001} by 
	showing that for every $\Delta \in \mathbb{Z}_{\geq 3}$ there exists a complex neighborhood of 
	the interval $\left[0, \frac{\left(\Delta - 1\right)^{\Delta - 1}}{\left(\Delta-2\right)^\Delta}\right)$ 
	on which the independence polynomial is nonzero for all graphs of maximum degree $\Delta$. Furthermore, 
	they gave an explicit neighborhood $U_\Delta$ containing this interval on which the independence 
	polynomial is nonzero for all finite rooted Cayley trees with branching number $\Delta$. The 
	question remained whether $U_\Delta$ would be zero-free for the independence polynomial of all 
	graphs of maximum degree $\Delta$. In this paper it is shown that this is not the case.
\end{abstract}

\maketitle

\begin{section}{Introduction}
	Let $G = (V,E)$ denote a simple graph. A subset of vertices $I \subseteq V$ is called 
	\emph{independent} if no two vertices $v_1, v_2 \in I$ are connected by an edge in $G$. We 
	define \emph{the independence polynomial} $Z_G$ as 
	\begin{equation}
  		\label{eq: Indep Pol}
  		Z_G(\lambda) = \hspace{-10 pt} \sum_{\substack{I \subseteq V \\ \text{independent}}} 
				\hspace{-5 pt} \lambda^{|I|}.
	\end{equation}
	In statistical physics the independence polynomial occurs as the partition function of the 
	hard-core model.

	For any $\Delta \in \mathbb{Z}_{\geq 3}$ we let $\mathcal{G}_\Delta$ be the set of graphs
	of maximum degree at most $\Delta$. It is interesting to study the location of the complex roots of $Z_G$ for 
	$G \in \mathcal{G}_\Delta$ from both a statistical physics perspective (see e.g. \cite{LeeYangPhase1},
	\cite{LeeYangPhase2} and \cite{Sokal2001}) 
	and a combinatorial perspective (see e.g. \cite{BarvinokBook2016}). Useful results in this area of research pertain to 
	finding regions in the complex plane for which $Z_G$ does not vanish for all $G \in \mathcal{G}_\Delta$.
	Patel and Regts \cite{PatelRegts2017} showed that such
	a zero-free domain for a partition function gives rise to a polynomial time 
	algorithm for approximating the function in that region. Their work is based on the interpolation
	method developed by Barvinok (see e.g. his book \cite{BarvinokBook2016}). Many
	results on zero-free regions regarding both the univariate 
	independence polynomial as stated in \eqref{eq: Indep Pol} and its multivariate generalization can 
	be found in \cite{ScottSokal}, \cite{BarvinokBook2016}, \cite{PetersRegts2017} and \cite{BencsCsikvari2018}.

	We will now state two results from \cite{PetersRegts2017} on the topic of zero-free regions 
	that are relevant to this paper. 

	\begin{theorem}[Theorem 1.1 in \cite{PetersRegts2017}]
		Let $\Delta \in \mathbb{Z}_{\geq 3}$ and let $\lambda_{\Delta} = 
		\frac{\left(\Delta - 1\right)^{\Delta - 1}}{\left(\Delta-2\right)^\Delta}$. There exists a 
		complex domain $D_\Delta$ containing the real interval $[0,\lambda_\Delta)$ such that 
		$Z_G(\lambda) \neq 0$ for all $G \in \mathcal{G}_\Delta$ and $\lambda \in D_\Delta$.
	\end{theorem}

	This result had previously been conjectured by Sokal \cite{Sokal2001}. We will henceforth denote 
	by $D_\Delta$ the maximal domain with the properties listed above.

	The other relevant result of \cite{PetersRegts2017} regards a zero-free region for the 
	independence polynomial of a certain subset of $\mathcal{G}_\Delta$, namely that of finite rooted Cayley trees. 
	These finite rooted Cayley trees have the following recursive definition. For each $\Delta$ we let 
	the $0$-th-level rooted tree with branching number $\Delta$ be the graph consisting of a single vertex 
	called the \emph{root}. We denote this tree by $T_{\Delta,0}$. For $n \geq 1$ we let $T_{\Delta,n}$ 
	denote the $n$-th-level rooted Cayley tree with branching number $\Delta$ and we define it by a single root 
	vertex attached to $\Delta - 1$ disjoint copies of $T_{\Delta,n-1}$ by their respective root vertices. 
	Note that for all $\Delta, n$ we have that $T_{\Delta,n} \in \mathcal{G}_\Delta$.

	\begin{theorem}[Proposition 2.1 in \cite{PetersRegts2017}]
  		\label{thm: zerofree Cayley}
  		Let $\Delta \in \mathbb{Z}_{\geq 3}$ and define
  		\begin{equation}
  			\label{eq: Ud}
	 		U_\Delta = \left\{\frac{-\alpha \cdot {(\Delta-1)}^{\Delta - 1}}
	 				{\left(\Delta - 1+ \alpha\right)^{\Delta}} : \left|\alpha\right| < 1 \right\}.
  		\end{equation}
  		Then
  		\begin{enumerate}
	 		\item 
	 		for all $n \in \mathbb{Z}_{\geq 0}$ and all $\lambda \in U_\Delta$ it is the case that
	 		$Z_{T_{\Delta,n}(\lambda)} \neq 0$;
	 		\item
	 		for any $\lambda \in \partial U_\Delta$ and neighborhood $U of \lambda$ there exists some 
	 		$n \in \mathbb{Z}_{\geq 0}$ and $\lambda' \in U$ such that $Z_{T_{\Delta,n}}(\lambda') = 0$.
		\end{enumerate}
	\end{theorem}
	
	In other words, $U_\Delta$ is a maximal zero-free region for the independence polynomials
	of rooted Cayley trees. From the  second part of Theorem \ref{thm: zerofree Cayley} it follows 
	that $D_\Delta \subseteq U_\Delta$. A natural question to pose is whether $D_\Delta = U_\Delta$. 
	This question appears as \emph{Question 2} in \cite{PetersRegts2017}. In this paper we show that 
	this is not the case.\footnote{This was first claimed by Juan Rivera-Letelier and
	Daniel \v{S}tefankovi\v{c} in personal communication.} We prove the following.

	\begin{theorem}
  		\label{thm: Main Theorem}
  		For $\Delta \in \{3,\dots,9\}$ there exist $\lambda \in U_\Delta$ with $G \in \mathcal{G}_\Delta$ 
  		such that $Z_G(\lambda) = 0$.
	\end{theorem}
	We will define a region $V_\Delta$ for which we get the inclusions 
	$D_\Delta \subseteq V_\Delta \subseteq U_\Delta$, and we will show that the latter inclusion 
	is strict for $3 \leq \Delta \leq 9$. The 
	definition of $V_\Delta$ is given in Section~\ref{sec: Concluding Remarks}. The other
	sections are dedicated to the proof of Theorem~\ref{thm: Main Theorem}.

	The main tool used in this paper comes from an area of complex dynamics that concerns the analysis
	of stable parameters of families of rational maps.
	\subsection*{Acknowledgment}
		The author would like to thank Han Peters and Guus Regts for useful discussions and advice.
		The author would also like to thank Ferenc Bencs for confirming some numerical 
		results.
\end{section}

\begin{section}{Setup and strategy}
	In this section we give the main definitions and results that we will use to prove 
	Theorem~\ref{thm: Main Theorem}. We will also outline the general strategy that the proof 
	follows. We start by defining the occupation ratio of a rooted tree and we analyze some of 
	its properties. Most definitions in the following subsection appear in \cite{PetersRegts2017}
	and are inspired by \cite{Weitz2006}.
	\begin{subsection}{Iteration of occupation ratios of rooted trees}
		Let $G = (V,E)$ denote a simple graph. For any $v \in V$ we define \emph{the closed 
		neighborhood $N[v]$ of $v$} as 
  		\[
	 		N[v] = \{u \in V: \{u,v\} \in E\} \cup \{v\}.
  		\]
		If $S \subseteq V$ we denote by $G[S]$ the subgraph of $G$ induced by the vertices in
		$S$. We denote the subgraph induced by the complement of $S$, i.e., $G[V\backslash S]$, 
		by $G\backslash S$. Finally, for any $v \in V$ we denote $G\backslash\{v\}$ by $G-v$. By 
		considering independent sets containing $v$ and not containing $v$ separately we obtain 
		the following recurrence relation of independence polynomials
  		\[
	 		Z_{G}(\lambda) = \lambda \cdot Z_{G\backslash N[v]}(\lambda) + Z_{G - v}(\lambda).
  		\]
		If $Z_{G - v}(\lambda) \neq 0$, we define the \emph{occupation ratio at $v$} as
  		\[
	 		R_{G,v}(\lambda)  	= \frac{Z_{G}(\lambda)}{Z_{G - v}(\lambda)} - 1 
	 							= \frac{\lambda \cdot Z_{G\backslash N[v]}(\lambda)}
	 									{Z_{G - v}(\lambda)}.
		\]
		We observe that for those $\lambda$ with $Z_{G - v}(\lambda) \neq 0$ we have that 
		$Z_G(\lambda) = 0$ if and only if $R_{G,v}(\lambda)~=~-1$. Now suppose that $T$ is 
		a tree with root vertex $v$ and $\lambda \in \mathbb{C}$ such that $Z_{T}(\lambda) 
		\neq 0$ and $Z_{T-v}(\lambda) \neq 0$. Define for $d \in \mathbb{Z}_{\geq 1}$ the larger 
		tree $\tilde{T}$ with a root vertex $\tilde{v}$ such that $\tilde{v}$ is attached to $d$ 
		copies of $T$ at their respective root vertices. Then 
		\begin{equation}
			\label{eq: Recurrence Ratios}
	 		R_{\tilde{T},\tilde{v}}(\lambda)	= \lambda \cdot \left(\frac{Z_{T-v}(\lambda)}
	 												{Z_{T}(\lambda)}\right)^d
												= \lambda \cdot \frac{1}{\left(1+R_{T,v}
													(\lambda)\right)^d}.
  		\end{equation}
		So if we define 
  		\[
	 		f_{\lambda,d}(z) = \frac{\lambda}{(1+z)^d},
  		\]
		we find that $R_{\tilde{T},\tilde{v}} = f_{\lambda,d}(R_{T,v}(\lambda))$. The occupation 
		ratio of a graph consisting of a single point is equal to $\lambda$. Therefore, to understand 
		whether $\lambda$ can occur as a zero of the independence polynomial of a finite Cayley tree 
		with branching number $\Delta$ it suffices to determine whether $-1$ appears in the orbit of 
		$\lambda$ under the map $f_{\lambda,\Delta-1}$. This analysis is done in \cite{PetersRegts2017}. 
		Instead of iterating with a single map we will consider iteration by a pattern of different maps 
		$f_{\lambda,d_1}, \dots, f_{\lambda, d_k}$, periodically applied. Effectively we will analyze the 
		roots of the independence polynomials of trees whose down degree is regular at every level.
	\end{subsection}
	\begin{subsection}{The rational semigroups $H_\Delta$}
		In the rest of this paper we will usually drop the subscript $\lambda$ from $f_{\lambda,d}$ and 
		write $f_d$ unless we want to stress a specific parameter $\lambda$. For $\Delta \in \mathbb{Z}_{\geq 3}$ 
		we define the rational semigroup $H_{\Delta}$ as the semigroup generated by $f_1, \dots f_{\Delta-1}$, 
		i.e,
		\[
  			H_\Delta = \left<f_{1},\dots, f_{\Delta-1}\right>.
		\]
		This semigroup consists of families of rational maps with the following property.

		\begin{lemma}
			\label{lem: g(0) = -1}
			Let $g \in H_{\Delta}$. If for some $\lambda \in \mathbb{C}$ we have $g_{\lambda}(0) = -1$. 
			Then there exists a tree $T \in \mathcal{G}_\Delta$ with $Z_T(\lambda) = 0$.
		\end{lemma}

		\begin{proof}
  			We can write $g = f_{d_n} \circ \cdots \circ f_{d_1}$. Let $k$ be the smallest positive 
  			integer such that $f_{\lambda, d_k} \circ \cdots \circ f_{\lambda, d_1}(0) = -1$. If $k = 1$, 
  			then $\lambda = f_{\lambda,d_1}(0) = -1$ and thus the statement is true since the independence 
  			polynomial of the graph consisting of one vertex is $\lambda + 1$.

  			If $k > 2$ then $\lambda \neq -1$. We let $T_0$ correspond to the empty graph and $T_1$ to 
  			the graph consisting of one root vertex $v_1$. Furthermore, we define for $m \in \{2,\dots, k\}$ 
  			the rooted tree $T_m$ as a root $v_m$ connected to $d_{m}$ copies of $T_{m-1}$ by their respective 
  			root vertices. Note that in this way $T_m \in \mathcal{G}_\Delta$ for all $m$. Also observe that 
  			$Z_{T_m-v_m}(\lambda) = \left(Z_{T_{m-1}}(\lambda)\right)^{d_m}$ and $Z_{T_m \backslash N[v_m]}(\lambda)
  			= \left(Z_{T_{m-2}}(\lambda)\right)^{d_m \cdot d_{m-1}}$. We will prove the following by induction. 
  			For $m \in \{2,\dots, k\}$ we have that
  			\begin{equation}
	 			\label{eq: induction proof}
	 				Z_{T_{l}}(\lambda) \neq 0 \text{ for $0 \leq l < m$} 
						\quad \text{ and } \quad
	 				R_{T_m, v_m} = \left(f_{\lambda, d_m} \circ \cdots \circ f_{\lambda, d_1}\right)(0).
  			\end{equation}
  			For $m = 2$ we have that $Z_{T_0}(\lambda) = 1 \neq 0$ and $Z_{T_1}(\lambda) = 1 + \lambda 
  			\neq 0$. As a result we find that $Z_{T_2 - v_2}(\lambda)$ and $Z_{T_2 \backslash N[v_2]}(\lambda)$
  			are not zero since they are powers of $Z_{T_1}(\lambda)$ and $Z_{T_0}(\lambda)$ respectively.
  			It follows that we can use \eqref{eq: Recurrence Ratios} to calculate the occupation ratio
  			of $T_2$ at $v_2$ by
  			\[
	 			R_{T_2,v_2}(\lambda) 	= f_{\lambda,d_2}(R_{T_1,v_1}) = f_{\lambda,d_2}(\lambda) 
		  								= \left(f_{\lambda,d_2}\circ f_{\lambda,d_1}\right)(0). 
  			\]
  			Now suppose that the statement in \eqref{eq: induction proof} is true for all values less 
  			than a certain $m > 2$. Then we know that $Z_{T_{m-1} - v_{m-1}}(\lambda) \neq 0$ and 
  			that $R_{T_{m-1},v_{m-1}}(\lambda) \neq -1$, which implies that $Z_{T_{m-1}}(\lambda) 
  			\neq 0$. This again implies that 
  			\[
	 			R_{T_m,v_m}(\lambda)  	= f_{\lambda,d_m}(R_{T_{m-1},v_{m-1}}(\lambda))
										= \left(f_{\lambda, d_m} \circ \cdots \circ 
													f_{\lambda, d_1}\right)(0).
  			\]
  			This proves the statement in \eqref{eq: induction proof}. Finally we can conclude that 
  			$R_{T_k,v_k}(\lambda) = -1$, while $Z_{T_k-v_k}(\lambda) \neq 0$. This implies that 
  			$Z_{T_k}(\lambda) = 0$, which concludes the proof since $T_k \in \mathcal{G}_\Delta$.
		\end{proof}
	\end{subsection}
	\begin{subsection}{Stable parameters of rational maps}
		This section contains the relevant results from the area of complex dynamics. The primary 
		object of study is that of the stable parameters of a holomorphic family of rational maps.
		The basis for this section is Chapter 4 of \cite{McMullen1994}. The result that we will state 
		follows from the $\lambda$-Lemma by Ma\~{n}\'{e}, Sad and Sullivan \cite{ManeSadSullivan1983}.

		Let $\CC = \mathbb{C} \cup \{\infty\}$ denote the Riemann sphere and let $\Omega \subseteq \mathbb{C}$
		denote a complex domain. We define a \emph{holomorphic family of rational maps}, parameterized 
		$\Omega$,as a holomorphic map $f: \Omega \times \CC \to \CC$ with the property that for 
		every $\lambda \in \Omega$ the map $z \mapsto f(\lambda,z)$ is a rational map. The first argument 
		of $f$ is thought of as a parameter and the map $\CC \to \CC: z \mapsto f(\lambda, z)$ is often 
		denoted by $f_\lambda$. Note that the elements of $H_\Delta$ are holomorphic families of rational 
		maps with respect to any complex domain. We will use the following definition to state the 
		subsequent theorem.
		\begin{definition}
			\label{def: Pers. Indiff.}
			Let $f$ be holomorphic family of rational maps and let $\lambda_0 \in \Omega$. We call a 
			periodic point $z$ of $f_{\lambda_0}$ with period $n$ \emph{persistently indifferent} if 
			there exists a neighborhood $U$ of $\lambda_0$ and a holomorphic map $w: U \to \CC$ such 
			that 
			\[
				w(\lambda_0) = z, 
					\quad 
				f_\lambda^n(w(\lambda)) = w(\lambda) 
					\quad \text{ and } \quad 
				|{(f_\lambda^n)}'(w(\lambda))| =1 
			\]
			for all $\lambda \in U$.
		\end{definition}
		\begin{theorem}[Part of Theorem 4.2 in \cite{McMullen1994}]
			\label{thm: Hol Motion}
			Let $f$ be a holomorphic family of rational maps parameterized by $\Omega$. And suppose 
			there exist holomorphic maps $c_i: \Omega \to \CC$ parameterizing the critical points of 
			$f$. Let $\lambda_0 \in \Omega$, then the following are equivalent.
			\begin{enumerate}
				\item 
				There is a neighborhood $U$ of $\lambda_0$ such that for all $\lambda \in U$ every 
				periodic point of $f_\lambda$ is either attracting, repelling or persistently indifferent.
				\item
				For all $i$ the families of maps given by
				\[
					\mathcal{F}_i = \{\lambda \mapsto f^n_\lambda(c_i(\lambda))\}_{n \geq 1}
				\]
				are normal at $\lambda_0$.
			\end{enumerate}
		\end{theorem}
		Our strategy will be to show that there are $g \in H_\Delta$ with $\lambda_0 \in U_\Delta$ such 
		that $g_{\lambda_0}$ has an indifferent fixed point that is not persistent. Then we will be able 
		to use non-normality of one of the critical points to show that arbitrarily close to $\lambda_0$ we 
		can find $\lambda$ for which we can derive a function $\tilde{g} \in H_\Delta$ with $\tilde{g}_\lambda(0) 
		= -1$. Then we will use Lemma~\ref{lem: g(0) = -1} to prove Theorem \ref{thm: Main Theorem}. This 
		will be made more precise in the next two sections.
	\end{subsection}
\end{section}

\begin{section}{Properties of the maps in $H_\Delta$}
	\begin{subsection}{The critical points}
		To apply Theorem~\ref{thm: Hol Motion} we need an understanding of the behaviour of the critical 
		points of the elements of $H_\Delta$. The following lemma states that for all $g \in H_\Delta$ 
		the critical points move locally holomorphically near all but finitely many $\lambda$.

		\begin{lemma}
			\label{lem: Critical Points}
			Let $g \in H_\Delta$ with $g = f_{d_k} \circ \cdots \circ f_{d_1}$. Let $\lambda_0 \in 
			\mathbb{C} - \{0\}$ be a parameter such that there are no indices $i,j$ with $1 \leq i < j \leq k$ 
			with 
			\begin{equation}
				\label{eq: assumption lemma critical points}
				\left(f_{\lambda_0, d_j} \circ \cdots \circ f_{\lambda_0, d_i}\right)\left(0 \right) = -1.
			\end{equation}
			Then there exists a neighborhood of $\lambda_0$ on which the critical points of $g$ can be 
			parameterized by holomorphic maps.
		\end{lemma}
		\begin{proof}
			For any $\lambda \in \mathbb{C} - \{0\}$ and $d \geq 2$ the critical points of $f_{\lambda,d}$ 
			are $-1$ and $\infty$. Therefore the critical points of $g_\lambda$ are given by points $z$ for 
			which there is some $i \in \{2,\dots,k\}$ with $d_i \geq 2$ and 
			\[
				\left(f_{\lambda,d_{i-1}} \circ \cdots \circ f_{\lambda,d_1}\right) (z) \in \{-1, \infty\},
			\]
			possibly together with $-1, \infty$ if $d_1\geq 2$. Since for any $d$ and nonzero $\lambda$ we have 
			that $f_{\lambda,d}(z) = \infty$ if and only if $z = -1$, we can write the critical points of 
			$g_\lambda$ as $X_\lambda = Y_\lambda \cup E$, where
			\[
				Y_\lambda = \hspace{-25 pt} \bigcup_{\substack{1 \leq i < k: \\ d_{i+1} \geq 2 \text{ or } 
									d_{i+2} \geq 2}} \hspace{-25 pt} \left\{z: f_{\lambda,d_i} \circ \cdots 
									\circ f_{\lambda,d_1} (z) = -1 \right\}
			\]
			and $E \subseteq \{-1,\infty\}$ with $\infty \in E$ only if $d_1 \geq 2$ and $-1 \in E$ only if 
			$d_1 \geq 2$ or $d_2 \geq 2$. Clearly the critical points in $E$ move holomorphically around 
			any neighborhood of $\lambda_0$ not containing $0$, since they do not depend on the parameter $\lambda$. 
			We will show that we can find a neighborhood of $\lambda_0$ on which the elements of $Y$ can also be 
			parameterized by holomorphic functions. Note that, since $f_{\lambda,d}(\infty) = 0$, it follows from 
			the assumption in (\ref{eq: assumption lemma critical points}) that $-1, \infty \not \in Y_{\lambda_0}$. 
			The Implicit Function Theorem guarantees that the elements of $Y$ move holomorphically near $\lambda_0$ 
			if for all $l$ and $z_0$, where $z_0$ is a solution to 
			\[
				\left(f_{\lambda_0,d_l} \circ \cdots \circ f_{\lambda_0,d_1}\right) (z_0) = -1,
			\]
			we have that
			\begin{equation}
				\label{eq: critical points req}
				\left(f_{\lambda_0,d_l} \circ \cdots \circ f_{\lambda_0,d_1}\right)'(z_0) \not \in \{0,\infty\}.
			\end{equation}
			To show that this is the case we first calculate that
			\[
				f_{\lambda,d}'(z) = -\frac{d}{1+z} \cdot f_{\lambda,d}(z),
			\]
			for all $\lambda,d$. We denote for all $i>0$ the $i$th element of the orbit of $z_0$ by $z_i$, i.e.,
			\[
				z_i = \left(f_{\lambda_0,d_i} \circ \cdots \circ f_{\lambda_0,d_1}\right) (z_0).
			\]
			Now we can write
			\[
				\left(f_{\lambda_0,d_l} \circ \cdots \circ f_{\lambda_0,d_1}\right)'(z_0) = \prod_{i=1}^l - 
					\frac{d_i \cdot z_i}{1 + z_{i-1}}. 
			\]
			The assumption of the lemma now guarantees that $\{-1,\infty,0\} \cap \{z_0, \dots, z_{l-1}\} = 
			\emptyset$ and since $z_l = -1$, we can conclude that the equation in (\ref{eq: critical points req}) 
			holds. The lemma now follows from an application of the Implicit Function Theorem.
		\end{proof}
		\begin{remark}
			\label{rem: critical points}
			Note that it follows from the proof that if $c$ is a holomorphic map parameterizing a critical 
			point of $g = f_{d_k} \circ \cdots \circ f_{d_1}$ on a domain $\Omega$ that either $c$ is 
			constantly $-1$ or $\infty$ on $\Omega$, or there is some index $l$ such that the holomorphic map
			\[
				\lambda \mapsto \left( f_{\lambda,d_l} \circ \cdots \circ f_{\lambda,d_1}\right)\left(c(\lambda)\right)
			\]
			is constantly $-1$. Since $-1$ gets mapped to $0$ in two applications of any two maps $f_d$, 
			independent of the degree of the individual maps and of $\lambda$, we get that there must be some 
			sequence of indices $d_{i_1}, \dots, d_{i_t} \in \{d_1, \dots, d_k\}$ such that
			\[
				g^3_\lambda(c(\lambda)) = \left(f_{\lambda,d_{i_t}} \circ \cdots \circ f_{\lambda,d_{i_1}}\right)
												\left(0\right)
			\]
			for all $\lambda \in \Omega$.
		\end{remark}	
	\end{subsection}
	\begin{subsection}{The indifferent fixed points}
		\label{sec: indifferent fixed points}
		To show that there are $g \in H_\Delta$ with $\lambda_0 \in U_\Delta$ such that $g_{\lambda_0}$ has 
		an indifferent fixed point that is not persistent we first show that there do no exist $g \in H_\Delta$ 
		and $\lambda_0 \in \mathbb{C}$ such that $g_{\lambda_0}$ has a persistently indifferent fixed point. 
		Note that we do not lose generality by considering only fixed points instead of periodic points since 
		$g \in H_\Delta$ implies that $g^N \in H_\Delta$ for any $N \in \mathbb{Z}_{\geq 1}$. The argument relies 
		on the following fact.
		\begin{lemma}
			\label{lem: persistently indifferent}
			Let $f$ be a holomorphic family of rational maps parameterized by a domain $\Omega$.
			Suppose that $\lambda_0 \in \Omega$ is a parameter such that $f_{\lambda_0}$ has a persistently 
			indifferent fixed point. Suppose also that on $\Omega$ we can write
			\begin{equation}
				\label{eq: rational function}
				f_{\lambda}(z) = \frac{p(\lambda,z)}{q(\lambda,z)},
			\end{equation}
			with $p,q \in \mathbb{C}[\lambda,z]$. Then the holomorphic family of rational maps $p/q$, where 
			the parameter plane is now taken to be the whole complex plane, has an indifferent fixed point 
			for all but finitely many parameters $\lambda \in \mathbb{C}$.
		\end{lemma}
		The proof of this lemma is elementary and can be found in the appendix. 

		Any $g \in H_\Delta$ can be written in the form displayed in (\ref{eq: rational function}). 
		A consequence of Lemma~\ref{lem: persistently indifferent} is now that if we can find a region 
		of parameters for which some $g \in H_\Delta$ has no indifferent fixed points, then we can conclude 
		that $g$ has no persistently indifferent fixed points for any parameter $\lambda$. We will prove that 
		this is the case for all $g \in H_\Delta$ by describing the fixed points of $g_\lambda$ for $\lambda$ 
		near $0$. These results are found in the next two lemmas.

		\begin{lemma}
			\label{lem: Attracting Fixed Point}
			Let $g \in H_\Delta$ and 
			$\lambda \in \mathbb{C}$ with $|\lambda| < \frac{(\Delta-1)^{\Delta-1}}{\Delta^{\Delta}}$. 
			Then $g_\lambda$ has an attracting fixed point.  
		\end{lemma}
		\begin{proof}
			Write $g = f_{d_k} \circ \cdots \circ f_{d_1}$ and let $B$ be a open 
			disc of radius $\frac{1}{\Delta}$ centered around zero. Then for any $d \in \{1, \dots, \Delta - 1\}$ 
			and $z \in \overline{B}$ we have 
			\[
				|f_{d,\lambda}(z)| = \frac{|\lambda|}{|(1+z)|^d} \leq \frac{|\lambda|}{(1-|z|)^{\Delta - 1}}
				< \frac{\frac{(\Delta-1)^{\Delta-1}}{(\Delta)^{\Delta}}}{\left(1-\frac{1}{\Delta}\right)^{\Delta-1}}
				= \frac{1}{\Delta}.
			\]
			This means that $B$ gets mapped strictly into itself by all the maps $f_{d_i,\lambda}$ and 
			thus also by $g_{\lambda}$. This means that $g_\lambda$ viewed as a map from $B$ to 
			itself is a strict contraction with respect to the Poincar\'e metric and thus $g_\lambda$
			is guaranteed to have an attracting fixed point inside $B$ by the Banach fixed point theorem.
		\end{proof}
		Note that it follows from the proof that the disc of radius $(\Delta-1)^{\Delta-1}/\Delta^{\Delta}$ is a 
		zero-free region of $Z_T$ for all trees $T \in \mathcal{G}_\Delta$, where the down degree is regular at every 
		level. Scott and Sokal show in \cite[Cor. 5.7]{ScottSokal} that this region remains zero-free for $Z_G$ for 
		all $G \in \mathcal{G}_\Delta$, even in the multivariate case (see also \cite{Shearer1985}).
		It turns out that we can also describe the repelling fixed points of elements in $H_\Delta$ for 
		all parameters inside this region.

		\begin{lemma}
			\label{lem: Repelling Fixed Point}
			Let $g \in H_\Delta$ and write $g = f_{d_k} \circ \cdots \circ f_{d_1}$.
			Let $\lambda \in \mathbb{C}- \{0\}$ with $|\lambda| < \frac{(\Delta-1)^{\Delta-1}}{\Delta^{\Delta}}$. 
			Then $g_\lambda$ has $d_1 \cdots d_k$ distinct repelling fixed points.
		\end{lemma}
		\begin{proof}
			For this proof denote $g = g_\lambda$.
			Let $B$ be an open disc of radius $\frac{\Delta-1}{\Delta}$ centered around $-1$. Let 
			$d \in \{1,\dots,\Delta-1\}$ and let $h(z) = \lambda/z^d$. Since $\overline{B}$ does not 
			intersect the positive real axis, we find that the inverse image $h^{-1}(B)$ consists of 
			$d$ disjoint domains $V_1, \dots, V_d$ such that for each $i$ the map $h|_{V_i}: V_i \to B$ 
			is a biholomorphism. Denote the inverse branches as $h^{-1}_1,\dots, h^{-1}_d$. Then for all 
			$z \in B$ all $i$ we have
			\[
				\left|h^{-1}_i(z)\right| = \left(\frac{|\lambda|}{|z|}\right)^{1/d} < 
					\left(\frac{\frac{(\Delta-1)^{\Delta-1}}{\Delta^{\Delta}}}{1-\frac{\Delta-1}{\Delta}}\right)^{1/d}
					=\left(\frac{\Delta-1}{\Delta}\right)^{(\Delta-1)/d} \leq \frac{\Delta-1}{\Delta}.
			\]
			The inverse branches of $f_{\lambda,d}$ on $B$ are given by $z \mapsto h^{-1}_i(z)-1$. If we 
			denote $U_i = h^{-1}_i(B)-1$, we see that $f_{\lambda,d}^{-1}(B) = U_1 \cup \cdots \cup U_d$, 
			$U_i \subsetneq B$ and $U_i \cap U_j = \emptyset$ for all $i,j$ with $i \neq j$. Furthermore, 
			$f_{\lambda,d}|_{U_i}$ is a biholomorphism for all $i$. By composition, we find that there are 
			$d_1 \cdots d_k$ inverse branches of $g$ on $B$, denoted by $g_1^{-1}, \dots, g^{-1}_{d_1 \dots, d_k}$ 
			with pairwise disjoint domains $W_1 \dots, W_{d_1\cdots d_k} \subsetneq B$ such that 
			$g^{-1}_i: B \to W_i$ is a biholomorphism for all $i$. Since $W_i$ is a strict subset of $B$ we find 
			that $g^{-1}_i$ is a strict contraction on $B$. Therefore, by the same reasoning as in Lemma 
			\ref{lem: Attracting Fixed Point}, $g^{-1}_i$ must have an attracting fixed point inside $W_i$. 
			This attracting fixed point of $g^{-1}_i$ is a repelling fixed point for $g$. Since every subset $W_i$ 
			contains such a point, we find that there are $d_1 \cdots d_k$ distinct repelling fixed points inside $B$.
		\end{proof}
		The previous three lemmas combined imply the following result.
		\begin{corollary}
			\label{cor: no persistently indifferent}
			Let $g \in H_\Delta$ be parameterized by some domain $\Omega$ and let $\lambda_0 \in \Omega$ such that 
			$g_{\lambda_0}$ has an indifferent fixed point. Then this fixed point is not persistently indifferent. 
		\end{corollary}
		\end{subsection}
\end{section}

\begin{section}{Proof of the main theorem}
	In this section we provide a proof for Theorem~\ref{thm: Main Theorem}. The essential idea is contained in 
	the following lemma.
	\begin{lemma}
		\label{lem: roots near indifferent point}
		Let $g \in H_\Delta$ not of the form $f_1^N$ and $\lambda_0 \in \mathbb{C}$ such that $g_{\lambda_0}$
		has an indifferent fixed point. Then for every neighborhood $U$ of $\lambda_0$ there exists a 
		$\lambda \in U$ and a tree $T \in \mathcal{G}_\Delta$ such that $Z_T(\lambda)=0$.
	\end{lemma}
	\begin{proof}
		Write $g = f_{d_k} \circ \cdots \circ f_{d_1}$. If there are indices $i,j$ with $1 \leq i < j \leq k$
		such that $\left(f_{\lambda_0, d_j} \circ \cdots \circ f_{\lambda_0, d_i}\right)\left(0 \right) = -1$,
		then we can apply Lemma~\ref{lem: g(0) = -1} on $f_{d_j} \circ \cdots \circ f_{d_i}$ to find that 
		there is a tree $T \in \mathcal{G}_\Delta$ such that $Z_T(\lambda_0) = 0$, so in this case the 
		statement is true. If these indices do not exist, then we apply Lemma~\ref{lem: Critical Points} to get a 
		domain $\Omega$ containing $\lambda_0$ on which the critical points of $g$ can be parameterized
		by holomorphic maps. Note that, since $g$ is not of the form $f_1^N$, $g$ has 
		critical points. By Corollary~\ref{cor: no persistently indifferent}, the indifferent fixed 
		point of $g_{\lambda_0}$ is not persistently indifferent and thus the first statement of 
		Theorem~\ref{thm: Hol Motion} is not fulfilled. Therefore there is at least one marked critical point 
		$c$ such that the family defined by 
		\[
			\left\{\lambda \mapsto g_\lambda^n(c(\lambda))\right\}_{n \geq 1}
		\] 
		is not normal around $\lambda_0$. From Remark~\ref{rem: critical points} it follows that
		there is some $h\in H_\Delta$ such that 
		\[
			\left\{\lambda \mapsto (g_\lambda^n\circ h_\lambda)(0)\right\}_{n \geq 1}
		\]
		is not normal around $\lambda_0$. Montel's Theorem now guarantees that 
		in any neighborhood $U$ of $\lambda_0$ there is a $\lambda \in U \cap \Omega$ and an 
		$N \in \mathbb{Z}_{\geq 3}$ such that $(g_{\lambda}^N \circ h_\lambda)(0) \in~\{0, \infty,-1\}$
		If $(g_{\lambda}^N \circ h_\lambda)(0) = -1$ we can directly apply Lemma~\ref{lem: g(0) = -1}
		to guarantee the existence of a tree $T \in \mathcal{G}_\Delta$ with $Z_T(\lambda) = 0$. 
		Otherwise, we remark that, since we have chosen $N \geq 3$, we can write 
		$g_{\lambda}^N \circ h_\lambda = f_{\lambda,\tilde{d}_2} \circ f_{\lambda, \tilde{d}_1} \circ 
		\tilde{g}_\lambda$ for some $\tilde{g} \in H_\Delta$ and $\tilde{d}_1,\tilde{d}_2 \in \mathbb{Z}_{\geq 1}$. 
		We find that
		$(f_{\lambda,\tilde{d}_2} \circ f_{\lambda, \tilde{d}_1} \circ \tilde{g}_\lambda)(0) = \infty$ implies 
		$(f_{\lambda, \tilde{d}_1} \circ \tilde{g}_\lambda)(0) = -1$ and $(f_{\lambda,\tilde{d}_2} \circ 
		f_{\lambda, \tilde{d}_1} 
		\circ \tilde{g}_\lambda)(0) = 0$ implies $\tilde{g}_\lambda(0) = -1$. In these cases we apply 
		Lemma~\ref{lem: g(0) = -1} to the respective maps to obtain the result.
	\end{proof}
	\begin{figure}[H]
  		\centering
    	\includegraphics[width=0.70\textwidth]{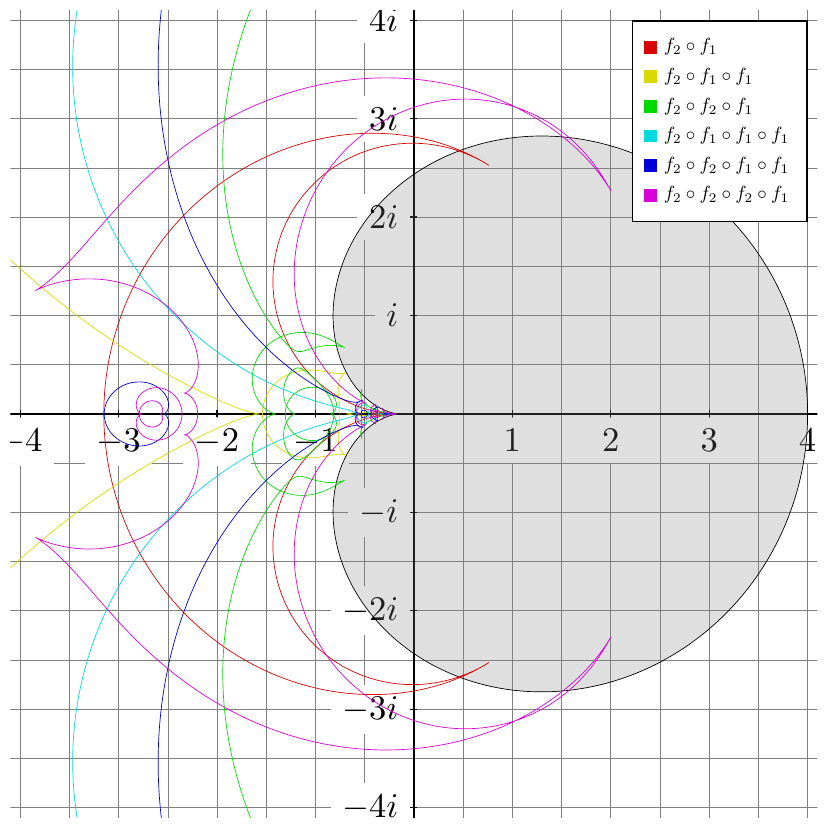}
  		\caption{The shaded area is $U_3$. Those parameters $\lambda$ for 
  			which $g_\lambda$ has an indifferent fixed point for different $g \in H_3$ are 
  			colored according to the legend.}
  		\label{fig: Indifferent Parameters}
	\end{figure}
	The remainder of the proof of Theorem~\ref{thm: Main Theorem} now consists of providing
	explicit examples of nontrivial $g \in H_\Delta$ with a parameter $\lambda \in U_\Delta$ such 
	that $g_\lambda$ has an indifferent fixed point for each $\Delta \in \{ 3, \dots, 9\}$. For 
	$\Delta = 3$ the degree is sufficiently small such that we can accurately calculate all such
	parameters for low degree $g \in H_3$, see Figure \ref{fig: Indifferent Parameters}. It is 
	immediately clear that there are many parameters that lie inside $U_3$.

	For larger $\Delta$ it quickly becomes intractable to calculate images like in 
	Figure~\ref{fig: Indifferent Parameters}, but it remains possible, given some 
	$g \in H_\Delta$, to accurately calculate those parameters $\lambda$ for which
	$g_\lambda$ has a parabolic fixed point of some given multiplier. Numerical 
	approximations for such parameters inside $U_\Delta$ where the multiplier 
	is $1$ are given in Table~\ref{tab: Numerical Results}. These results 
	prove Theorem~\ref{thm: Main Theorem}.
	\begin{table}[H]
		\caption{Given for each $\Delta \in \{3,\dots,9\}$ is a $g \in H_\Delta$
			together with an approximation of a $\lambda \in U_\Delta$ such that $g_\lambda$
			has a fixed point $z$ with $g_\lambda'(z) = 1$. An approximation is given for 
			the absolute value of $\alpha$, where $\alpha$ is a solution to 
			$\lambda = \frac{-\alpha \cdot {(\Delta-1)}^{\Delta - 1}}
			{\left(\Delta - 1+ \alpha\right)^{\Delta}}$
			of least absolute value. This value being less than $1$ confirms that 
			$\lambda \in U_\Delta$, see (\ref{eq: Ud}).}
		\begin{tabular}{cccc}
			\hline
			$\Delta$ & $g$  & $\lambda$  & $|\alpha|$ \\
			\hline
		    $3$ 	& $f_2 \circ f_1$ 			& $0.7624680 	+ 2.5253695\, i$  & $0.97581$ \\
		    $4$		& $f_3 \circ f_1$ 			& $0.37725715 	+ 1.21796118\, i$ & $0.99987$ \\
		    $5$   	& $f_4 \circ f_4 \circ f_1$	& $-0.24803954 	+ 0.17613988\, i$ & $0.98607$ \\
		    $6$		& $f_5 \circ f_5 \circ f_1$	& $-0.19657017 	+ 0.14664968\, i$ & $0.99630$ \\
		    $7$		& $f_6 \circ f_6 \circ f_2$	& $-0.15604600 	+ 0.14898604\, i$ & $0.97830$ \\
		    $8$		& $f_7 \circ f_7 \circ f_2$	& $-0.13276176 	+ 0.12728769\, i$ & $0.98408$ \\
		    $9$		& $f_8 \circ f_8 \circ f_2$	& $-0.11587455 	+ 0.11090067\, i$ & $0.98967$ \\
    	\hline
 		\end{tabular}
 		\label{tab: Numerical Results}
	\end{table}
\end{section}
\begin{section}{Concluding remarks}
	\label{sec: Concluding Remarks}
	It follows from Lemma~\ref{lem: roots near indifferent point} that the set of roots of $Z_G$ for 
	all $G \in \mathcal{G}_\Delta$ accumulates at the boundary of $V_\Delta$, where
	\[
		V_\Delta = \left\{\lambda: \text{$g_\lambda$ has exactly 1 attracting fixed point for all $g \in H_\Delta$
		with $g \neq f_1^N$}\right\}.
	\]
	Recall that we defined $D_\Delta$ to be the largest domain containing $0$ that is zero-free for 
	all $G \in \mathcal{G}_\Delta$. In Section~\ref{sec: indifferent fixed points} we showed that 
	parameters $\lambda$ with $|\lambda| < (\Delta-1)^{\Delta-1}/\Delta^{\Delta}$ lie in $V_\Delta$. In
	\cite{ScottSokal} it is shown that these $\lambda$ also lie in $D_\Delta$. It follows that 
	$D_\Delta \subseteq V_\Delta$. By definition, we have $V_\Delta \subseteq U_\Delta$ and thus we 
	can write
	\[
		D_\Delta \subseteq V_\Delta \subseteq U_\Delta,
	\]
	where the last inclusion was shown to be strict for $3 \leq \Delta \leq 9$ in this paper.
	Two obvious questions that remain open are whether $V_\Delta \neq U_\Delta$ for $\Delta \geq 10$
	and whether $D_\Delta = V_\Delta$ for any $\Delta$. 

	Another question concerns the computational 
	complexity of approximating the independence polynomial. Recall that for $\lambda \in D_\Delta$ there is an
	polynomial time algorithm to approximate $Z_G(\lambda)$ for $G \in \mathcal{G}_\Delta$ (see 
	\cite{PatelRegts2017}). On the other hand, for non-real $\lambda$ outside 
	$U_\Delta$ it was shown by Bez\'{a}kov\'{a}, Galanis, Goldberg and \v{S}tefankovi\v{c} \cite{GoldbergEtAl2018} 
	that approximating $Z_G(\lambda)$ is \#P-hard. The computational complexity 
	of approximating $Z_{G}(\lambda)$ for $\lambda \in U_\Delta - V_\Delta$
	remains to be studied. Given the similar definitions of the region $U_\Delta$ and $V_\Delta$, one 
	might expect that approximating $Z_G(\lambda)$ for non-real $\lambda$ outside $V_\Delta$ is also
	\#P-hard.
\end{section}

\begin{appendix}
\begin{section}{Proof of Lemma \ref{lem: persistently indifferent}}
	The proof that we present here is algebraic rather than analytic in nature. We view $\mathbb{C}[\lambda,z]$
	as a subring of the ring $\mathbb{C}(\lambda)[z]$. This ring is Euclidian, so in particular it is a unique 
	factorization domain. Therefore we can state the following simple lemma.
	\begin{lemma}
		\label{lem: Common Roots}
		Let $p,q \in \mathbb{C}[\lambda,z]$ be coprime in $\mathbb{C}(\lambda)[z]$. Then there are only 
		finitely many $\lambda \in \mathbb{C}$ such that the polynomials $p(\lambda,z)$ and $q(\lambda,z)$
		viewed as elements of $\mathbb{C}[z]$ have common roots.
	\end{lemma}
	\begin{proof}
		Since $p,q$ are coprime in the Euclidian domain $\mathbb{C}(\lambda)[z]$, there exist elements
		$u,v \in \mathbb{C}(\lambda)[z]$ such that $u \cdot p + v \cdot q = 1$. There exists an element 
		$w \in \mathbb{C}[\lambda]$ such that the coefficients of $w \cdot u$ and $w \cdot v$ are elements
		of $\mathbb{C}[\lambda]$. It follows that for all $\lambda,z$ we have can write down the following 
		equality of polynomials
		\[
			w(\lambda) u(\lambda,z) \cdot p(\lambda,z) + w(\lambda) v(\lambda,z) \cdot q(\lambda,z)
				=
			w(\lambda).
		\]
		We deduce now that if there is some pair $(\lambda_0,z_0)$ that is both a root of $p$ and $q$, then 
		$\lambda_0$ is a root of $w$. Since $w$ has only finitely many roots, we deduce that only finitely
		many such $\lambda$ can exist. 
	\end{proof}
	Before we prove Lemma~\ref{lem: persistently indifferent}, we recall some properties of the algebraic
	construction called the \emph{resultant}. Namely, if $k$ is a field and $f,g \in k[x]$, then the resultant
	of $f$ and $g$, denoted
	by $\Res_x(f,g)$, is an integer polynomial in the coefficients of $f$ and $g$ with the property that
	$\Res_x(f,g) = 0$ if and only if $f,g$ have a common factor in $k[x]$.
	One can read about the theory of resultants in many introductory texts on computational algebraic geometry, 
	see e.g. \cite[Chapter 3, \textsection 5]{CoxEtAl2007}. We now present a proof of 
	Lemma~\ref{lem: persistently indifferent}.
	\begin{proof}[Proof of Lemma \ref{lem: persistently indifferent}]
		We can assume that $p,q$ are coprime in $\mathbb{C}(\lambda)[z]$. Let $U$ be a neighborhood
		of $\lambda_0$ together with a map $w: U \to \CC$ that has the properties described in 
		Definition~\ref{def: Pers. Indiff.}. We can assume that the map $w$ avoids $\infty$. 
		The holomorphic map 
		\[
			\lambda \mapsto f_\lambda'(w(\lambda))
		\]
		is an open map with constant absolute value and is thus constant on $U$, say equal to $\alpha$
		with $|\alpha| = 1$. Note that we can write 
		\[
			 f_\lambda'(z) = \frac{s(\lambda,z)}{t(\lambda,z)},
		\]
		with $s,t \in \mathbb{C}[\lambda,z]$ coprime in $\mathbb{C}(\lambda)[z]$. Define the following 
		polynomials 
		\[
			l(\lambda,z) = p(\lambda,z) - z \cdot q(\lambda,z)
				\quad \text{ and } \quad
			m(\lambda,z) = s(\lambda,z) - \alpha \cdot t(\lambda,z).	
		\]
		It follows from Lemma~\ref{lem: Common Roots} that for all but finitely many $\lambda$
		we have that $f_\lambda(z) = z$ if and only if $l(\lambda,z) = 0$ and similarly for all but finitely
		many $\lambda$ we have $f_\lambda'(z) = \alpha$ if and only if $m(\lambda,z) = 0$. Consider
		the polynomial
		\[
			R(\lambda) = \Res_z(l,m).
		\]
		Note that $R(\lambda) \in \mathbb{C}[\lambda]$. Since for all but finitely many $\lambda \in U$
		the polynomials $m(\lambda,z)$ and $l(\lambda,z)$ have a common root, namely $w(\lambda)$,
		we find that $R(\lambda)$ has infinitely many roots and is constantly $0$ as a result. This means 
		that for all $\lambda \in \mathbb{C}$ the polynomials $m(\lambda,z)$ and $l(\lambda,z)$ have 
		a common root. This again means that for all but finitely many $\lambda$ there is some $z \in \mathbb{C}$
		such that $f_\lambda(z) = z$ and $f'_\lambda(z) = \alpha$, where we now 
		consider $f_\lambda$ to be defined for every complex parameter $\lambda$. This concludes the proof.
	\end{proof}

\end{section}

\end{appendix}

\bibliography{mainbib}{}
\bibliographystyle{amsalpha}

\end{document}